\documentclass[12pt, a4paper]{article}

\usepackage[utf8]{inputenc}
\usepackage[T1]{fontenc}
\usepackage[english]{babel}

\usepackage[margin=3cm]{geometry}

\usepackage[dvipsnames]{xcolor}
\usepackage[shortlabels]{enumitem} 
\usepackage{booktabs} 
\usepackage{multirow} 

\usepackage{mathtools, amssymb}
\usepackage{mathrsfs} 

\usepackage{amsthm}
\allowdisplaybreaks
\usepackage{eucal} 
\theoremstyle{plain}
\newtheorem{thm}{Theorem}[section]

\newtheorem{lem}[thm]{Lemma}
\newtheorem{lemma}[thm]{Lemma}
\newtheorem{prop}[thm]{Proposition}
\newtheorem{proposition}[thm]{Proposition}

\newtheorem*{thm*}{Theorem}
\theoremstyle{definition}

\newtheorem{ex}[thm]{Example}
\newtheorem{example}[thm]{Example}
\theoremstyle{remark}
\newtheorem{rmk}[thm]{Remark}

\usepackage{tikz-cd}
\newcommand{\drpbk}{\ar[dr, phantom, "\lrcorner", very near start]}
\newcommand{\dlpbk}{\ar[dl, phantom, "\llcorner", very near start]}

\usepackage[all]{xy}
\newcommand{\cd}[2][]{\vcenter{\hbox{\xymatrix#1{#2}}}}

\newcommand{\ds}[1]{\mathbf{#1}} 

\providecommand{\kat}[1]{\text{\textbf{\textsl{#1}}}}
\newcommand{\Set}{\kat{Set}}

\newcommand{\Grpd}{\kat{Grpd}}
\newcommand{\grpd}{\kat{grpd}}

\newcommand{\simplexcat}{\boldsymbol \Delta}

\newcommand{\LIN}{\kat{LIN}}

\newcommand{\F}{\mathbb{F}} 
\newcommand{\B}{\mathbb{B}} 
\newcommand{\I}{\mathbb{I}} 
\newcommand{\Sur}{\mathbb{S}} 
\newcommand{\Sp}{\mathbb{S}_p} 

\newcommand{\C}{\mathcal{C}}
\newcommand{\D}{\mathcal{D}}

\newcommand{\N}{\mathbb{N}}

\newcommand{\Q}{\mathbb{Q}}

\DeclareMathOperator{\Map}{Map}

\DeclareMathOperator{\Aut}{Aut}

\DeclareMathOperator{\Id}{Id}
\DeclareMathOperator{\id}{id}

\DeclareMathOperator{\Dec}{Dec}

\newcommand{\xra}{\xrightarrow}

\newcommand{\surj}{\twoheadrightarrow}

\newcommand{\eq}{\simeq}

\newcommand{\name}[1]{\ulcorner #1\urcorner}
\newcommand{\fatnerve}{\mathbf{N}}
\newcommand{\Smonad}{\mathsf{S}}
\newcommand{\SSS}{\mathsf{S}}
\newcommand{\tensor}{\otimes}
\newcommand{\thg}{{\mathord{\text{--}}}}
\newcommand{\fibre}{\phi}

\newcommand{\op}{^{\text{{\rm{op}}}}}

\renewcommand{\emptyset}{\varnothing}

\DeclarePairedDelimiter\card{\lvert}{\rvert}

\newcommand{\arxiv}[1]{\href{http://arxiv.org/abs/#1}{arxiv:#1}}

\usepackage{hyperref}
\hypersetup{
    pdftitle={Hereditary species as monoidal decomposition spaces, comodule bialgebras, and operadic categories},
    pdfauthor={Louis Carlier},
    colorlinks=true,
    urlcolor=teal, 
    linkcolor=RoyalBlue, 
    citecolor=ForestGreen, 
    breaklinks=true
}

\author{Louis Carlier}
\title{Hereditary species as monoidal decomposition spaces, comodule bialgebras, and operadic categories}
\date{}

\newcommand{\address}{{%
  \bigskip
  \footnotesize
  \textsc{Departament de Matemàtiques \\
  \indent Universitat Autònoma de Barcelona}\par\nopagebreak 
  \textit{E-mail address}: \href{mailto:louiscarlier@mat.uab.cat}{\nolinkurl{louiscarlier@mat.uab.cat}}
}}

\begin{document}

\maketitle

\begin{abstract}
We show that Schmitt's hereditary species induce monoidal decomposition spaces, and exhibit Schmitt's bialgebra construction as an instance of the general bialgebra construction on a monoidal decomposition space.
We show furthermore that this bialgebra structure coacts on the underlying restriction-species bialgebra structure so as to form a comodule bialgebra.
Finally, we show that hereditary species induce a new family of examples of operadic categories in the sense of Batanin and Markl~\cite{BM}. 
\end{abstract}

\phantomsection
\addcontentsline{toc}{section}{Introduction}
\section*{Introduction}
\label{sec:intro}

The notion of decomposition space was introduced in combinatorics by Gálvez-Carrillo, Kock, and Tonks~\cite{GKT1, GKT2} as a generalisation of Möbius categories~\cite{Leroux76} in turn generalising locally finite posets~\cite{Rota} and Cartier–Foata finite-decomposition monoids~\cite{CF}, for the purpose of incidence (co)algebras and Möbius inversion. 
The same notion was introduced by Dyckerhoff and Kapranov~\cite{DK} under the name unital 2-Segal space, for use in homological algebra and representation theory.
Decomposition spaces are simplicial groupoids that satisfy an exactness condition weaker than the Segal condition: while the Segal condition essentially characterises categories (the ability to compose), the decomposition space axiom expresses the ability to decompose.

Not all coalgebras in combinatorics arise from posets, monoids or categories. In a broader perspective important examples come from the Waldhausen $S$-construction in topology and Hall algebras of various flavours. In combinatorics, an important class of coalgebras are the coalgebras of restriction species in the sense of Schmitt~\cite{Schmitt} (see also \cite[\S 8.7]{AM}), and the more general notion of directed restriction species of \cite{GKT:restrict}, shown to be examples of decomposition spaces.  
Examples of these notions are given by the chromatic Hopf algebra of graphs and the Butcher--Connes--Kreimer Hopf algebra of rooted trees, as well as many related structures such as matroids, posets (talking here about a Hopf algebra of all posets, not just a coalgebra of an individual poset).  Most of these examples are decomposition spaces that do not satisfy the Segal condition.

In these examples, the comultiplication is given by splitting some structure into two parts of equal status, constituting then the left and right tensor factors.  However, many important Hopf algebras in combinatorics are more asymmetric, having on one side of the comultiplication a monomial instead of a linear tensor factor.
In the Segal case, the comultiplications with both tensor factors linear are typical for incidence coalgebras of categories, whereas the comultiplications with a monomial in the left-hand tensor factor are typical for operads. Indeed, incidence coalgebras and bialgebras of operads are another important class covered by the decomposition space framework, see \cite{GKT:combinatorics} and \cite{KW}.

Just as there are many linear-linear coalgebras that do not come from categories, there are important examples of multilinear-linear coalgebras that do not come from operads. An important class of such coalgebras is given by Schmitt's hereditary species~\cite{Schmitt}. These are structures that admit restriction (like restriction species) but also admit induction along quotient maps. Formally these are functors $H\colon \Sp \to \Set$, where $\Sp$ denotes the category of finite sets and partially defined surjections. The induced comultiplication works by summing over all partitions of the underlying set, and then putting the monomial of all the blocks (with restricted structure) on the left and putting the quotient structure on the right. Schmitt~\cite{Schmitt} identified the properties needed for this to define a coassociative coalgebra (in fact a bialgebra, and most often a Hopf algebra) and exhibited important examples, such as in particular the hereditary species of simple graphs.

The present paper shows that every hereditary species constitutes an example of a monoidal decomposition space, and that Schmitt's bialgebra construction is a special case of the general incidence bialgebra construction for monoidal decomposition spaces.
These decomposition spaces are generally \emph{not} Segal
spaces (see Remark~\ref{notSegal}), and can therefore be seen as the first class of examples of decomposition spaces filling the missing entry in the following table.
\begin{center}
\begin{tabular}{@{}lll@{}}
\toprule
                &     \emph{linear-linear}         &  \emph{multilinear-linear} \\ \midrule
    \emph{Segal-type}  & posets, monoids, categories    &     operads \\ \midrule
 \multirow{2}{*}{\emph{non-Segal type}} & restriction species, $S$-construction,  & \\
 & Hall algebras & \\
 \bottomrule
\end{tabular}
\end{center}
The construction is similar to the two-sided bar construction
for operads, see \cite{KW}.

\medskip

After a brief review in Section~\ref{sec:grpd&ds} of needed notions about groupoids and decomposition spaces, Section~\ref{sec:hersp} summarises Schmitt's hereditary species.
In Section~\ref{sec:dssurj}, we realise the hereditary species of finite sets as a Segal space $\ds{S}$ (hence a decomposition space) and establish some finiteness conditions.

\medskip
\noindent
{\bf Proposition~\ref{SSSegal} and \ref{Xfinite}.} {\em
The pseudosimplicial groupoid $\ds{S}$ is Segal, and is complete, locally finite, locally discrete, and of locally finite length.
}

\medskip

In Section~\ref{sec:hersp&ds}, exploiting the hereditary species of finite sets, we show that every hereditary species $H$ defines a monoidal decomposition space $\ds{H}$, and hence a bialgebra at the groupoid-slice level, and we recover Schmitt's construction by taking homotopy cardinality.

\medskip
\noindent
{\bf Proposition~\ref{Hds} and \ref{Schmittcoincide}.} {\em
For every hereditary species $H$, the simplicial groupoid $\ds H$ is a monoidal decomposition space.
The incidence bialgebra obtained by taking homotopy cardinality
coincides with the Schmitt bialgebra associated to $H$.
}

\medskip

Every hereditary species is also a restriction species, and the free algebra on its incidence coalgebra is therefore a bialgebra.
In Section~\ref{sec:comodulebialgebra}, we show that the incidence bialgebra of a hereditary species coacts on this bialgebra, so as constitute a \emph{comodule bialgebra}. 

\medskip
\noindent
{\bf Proposition~\ref{prop:comodulebialg}.} {\em
  The hereditary-species bialgebra $B$ coacts on the restriction-species bialgebra $A$, so as to make $A$ a left comodule bialgebra over $B$.
}

\medskip

Comodule bialgebras have been found important recently in numerical analysis~\cite{CEFM} and in stochastic analysis~\cite{BHZ}, and there are general constructions based on operads and trees \cite{Foissy:operads}, \cite{Kock:BD-bis}.
The incidence comodule bialgebras of hereditary species introduced here constitute a new general class of comodule bialgebras, not related to trees.

In Section~\ref{sec:operadic}, we describe a different construction on hereditary species, showing that simple hereditary species induce operadic categories in the sense of Batanin and Markl~\cite{BM}. Precisely we define a functor from simple hereditary species to operadic categories. This is interesting in its own right, as it constitutes a new family of examples of operadic categories which had not been observed before. The construction is not directly related to decomposition spaces, but suggests that further connections are to be discovered.

\paragraph{Acknowledgements}
The author would like to thank Joachim Kock for his advice and support all along the project. The author was supported by a PhD grant attached to MTM2013-42293-P, and by the FEDER-MINECO grant MTM2016-80439-P from the Spanish Ministry of Economy and Competitiveness.

\section{Groupoids and decomposition spaces}\label{sec:grpd&ds}

We work in the $2$-category $\Grpd$ of groupoids, maps, and homotopies.
All the notions will be invariant under homotopy equivalences. We usually omit the term homotopy and say `pullback', `fibre', etc. instead of `homotopy pullback', `homotopy fibre', etc.
The squares are commutative up-to-homotopy, they come equipped with an (invertible) $2$-cell.

\paragraph{Pullbacks and fibres}
    The main tool used throughout this paper are (homotopy) pullbacks. We use the following lemmas many times.

\begin{lem}[{\cite{CK:HTCG},\cite[Lemma 4.4.2.1]{Lurie}}]\label{prismlemma}
  Given a prism diagram of groupoids
  \begin{center}
    \begin{tikzcd}
        X   \ar[r, ""] \ar[d, ""]
          & X' \ar[r, ""]\ar[d, ""] \drpbk
          & X''\ar[d, ""]\\
        Y \ar[r, ""'] & Y' \ar[r, ""] & Y''
    \end{tikzcd}
  \end{center}
    in which the right-hand square is a pullback. 
    Then the outer rectangle is a pullback if and only of the left-hand square is.
\end{lem}

Given a map of groupoids $p\colon X \to S$ and an object $s \in S$, the \emph{fibre} $X_s$ of $p$ over $s$ is the pullback
\begin{center}
    \begin{tikzcd}
        X_s \ar[r, ""] 
            \ar[d, ""'] 
            \drpbk
             & X \ar[d, "p"]\\
        1 \ar[r, "\name{s}"'] & S.
    \end{tikzcd}
\end{center}


\begin{lem}[\cite{CK:HTCG}]\label{fibreslemma}
  A square of groupoids
\begin{center}
    \begin{tikzcd}
        P \ar[r, "u"] \ar[d] & Y \ar[d] \\
        X \ar[r] & S
    \end{tikzcd}
\end{center}
 is a pullback if and only if for each $x\in X$ the induced comparison map $u_x\colon P_x \to Y_{fx}$ is an equivalence.
\end{lem}

\paragraph{Slices and linear functors}
Recall that the objects of the slice category $\Grpd_{/I}$ are maps of groupoids with codomain $I$.
Pullback along a morphism $f\colon J \to I$, defines an functor $f^*\colon \Grpd_{/I} \to \Grpd_{/J}$. This functor is right adjoint to the functor $f_!\colon \Grpd_{/J} \to \Grpd_{/I}$ given by postcomposing with $f$.
A span
$I \xleftarrow{p} M \xrightarrow{q} J$
induces a functor between the slices by pullback and postcomposition
\[\Grpd_{/I} \xra{p^*} \Grpd_{/M} \xra{q_!} \Grpd_{/J}.\]
A functor is \emph{linear} if it is homotopy equivalent to a functor induced by a span.
The following Beck-Chevalley rule holds for groupoids:
for any pullback square
\begin{center}
  \begin{tikzcd}
      J \ar[r, "f"] 
         \ar[d, "p"'] 
         \drpbk
         & I \ar[d, "q"]\\
      V \ar[r, "g"'] & U,
  \end{tikzcd}
\end{center}
the functors
$
    p_!f^*, g^*q_!\colon \Grpd_{/I} \to \Grpd_{/V}
$
are naturally homotopy equivalent (see \cite{GHK} for the technical details regarding coherence of these equivalences in the $\infty$-groupoids setting).
By the Beck-Chevalley rule, the composition of two linear functors is linear.
The slice groupoid $\Grpd_{/I}$ plays the role of the vector space with basis $I$.

We denote by $\LIN$ the monoidal $2$-category of slice categories  and linear functors, with the tensor product induced by the cartesian product:
\[
    \Grpd_{/I} \otimes \Grpd_{/J} := \Grpd_{/ I\times J},
\]
with neutral object $\Grpd \eq \Grpd_{/1}$.
For an extended treatment of linear functors and homotopy linear algebra, we refer to \cite{GKT:HLA}.

\paragraph{Decomposition spaces and incidence coalgebras}
The notion of decomposition space was introduced by Gálvez-Carrillo, Kock, and Tonks \cite{GKT1}, in the general setting of simplicial $\infty$-groupoids, and independently by Dyckerhoff and Kapranov \cite{DK} under the name unital $2$-Segal space.
The natural level of generality for decomposition spaces in combinatorics is that of simplicial groupoids, because many combinatorial objects have symmetries, which are taken care of by the groupoid formalism. For a survey motivated by combinatorics, see \cite{GKT:combinatorics}.

    The simplex category $\simplexcat$ has an active-inert factorisation system. An arrow is active when it preserves endpoints, and is inert if it is distance preserving.
    In the present contribution, a \emph{decomposition space} $X\colon \simplexcat\op \to \Grpd$ is a simplicial groupoid such that the image of any active-inert pushout in $\simplexcat$ is a pullback of groupoids.
It is enough to check that the following squares are pullbacks, where $0 \le k \le n$:
  \begin{center}
    \begin{tikzcd}
        X_{n+1}   \ar[r, "s_{k+1}"] 
                  \ar[d, "d_\bot"']
                  \drpbk
          & X_{n+2} \ar[d, "d_\bot"] \\
        X_n \ar[r, "s_k"'] & X_{n+1},
    \end{tikzcd} \!\!\!
    \begin{tikzcd}
       X_{n+2} \ar[d, "d_\bot"] 
          & X_{n+3} \ar[l, "d_{k+2}"']
                    \ar[d, "d_\bot"]
                    \dlpbk \\
         X_{n+1}  & X_{n+2} \ar[l, "d_{k+1}"],
    \end{tikzcd} \!\!\!
    \begin{tikzcd}
        X_{n+1}   \ar[r, "s_k"] 
                  \ar[d, "d_\top"']
                  \drpbk
          & X_{n+2} \ar[d, "d_\top"]  \\
        X_n \ar[r, "s_k"'] & X_{n+1},
    \end{tikzcd} \!\!\!
    \begin{tikzcd}
        X_{n+2} \ar[d, "d_\top"] 
          & X_{n+3} \ar[l, "d_{k+1}"']
                    \ar[d, "d_\top"]
                    \dlpbk \\
         X_{n+1}  & X_{n+2} \ar[l, "d_{k+1}"].
    \end{tikzcd}
  \end{center}

Decomposition spaces can be seen as an abstraction of posets. 
The pullback condition is precisely the condition required to obtain a counital coassociative comultiplication on $\Grpd_{/X_1}$, called the \emph{incidence coalgebra}. 
Precisely, the span
\[
    X_1 \xleftarrow{d_1} X_2 \xrightarrow{(d_2,d_0)} X_1 \times X_1,
\]
defines a linear functor, the \emph{comultiplication}:
\begin{align*}
    \Delta\colon  \Grpd_{/X_1} & \to  \Grpd_{/X_1 \times X_1} \\
            (T \xra{t} X_1) & \mapsto (d_2,d_0)_{!} \circ d_1^{*}(t).
\end{align*}
The span 
$X_1 \xleftarrow{s_0} X_0 \xrightarrow{z} 1$
defines a linear functor, the \emph{counit}:
\begin{align*}
    \varepsilon\colon  \Grpd_{/X_1} & \to  \Grpd \\
            (T \xra{t} X_1) & \mapsto z_{!} \circ s_0^{*}(t).
\end{align*}
We obtain a coalgebra $(\Grpd_{/X_1}, \Delta, \varepsilon)$ called the \emph{incidence coalgebra}.

\begin{prop}[{\cite[Proposition 2.3.3]{DK}}, {\cite[Proposition 3.5]{GKT1}}]
    Every Segal space is a decomposition space. In particular, the nerve of a poset is a decomposition space.
\end{prop}



\paragraph{Homotopy cardinality \cite{GKT:combinatorics}}

In order to be able to take homotopy cardinality to get a coalgebra at the numerical level, we need to assume some finiteness conditions.
A groupoid is \emph{locally finite} if the automorphism groups are finite, and is \emph{finite} if furthermore it has finitely many components. We denote by $\grpd$ the category of finite groupoids.
Note that every set is locally finite.
A map of groupoids is \emph{finite} if each fibre is finite.

The \emph{homotopy cardinality} of a finite groupoid $X$ is defined to be 
\[
    |X| = \sum_{x\in \pi_0 X} \frac{1}{|\text{Aut}(x)|} \in \Q,
\]
and the homotopy cardinality of a finite map of groupoid is
\[
    |X\to S| = \sum_{s\in \pi_0 S} \frac{|X_s|}{|\text{Aut}(s)|} \delta_s \in \Q_{\pi_0 S},
\]
where $X_s$ is the homotopy fibre,
and $\Q_{\pi_0 S}$ is the vector space spanned by iso-classes, denoted by the formal symbol $\delta_s$, for $s \in \pi_0 S$.
Homotopy equivalent groupoids have the same cardinality.
For a locally finite groupoid, there is a notion of cardinality $|\thg|\colon \Grpd_{/S} \to \Q _{\pi_0 S}$ sending a basis element $\name{s}$ to the basis element $\delta_s = |\name{s}|$.

A decomposition space $X$ is \emph{locally finite} if $X_1$ is a locally finite groupoid, and the two maps $s_0$ and $d_1$ have finite fibres. 
 If we require furthermore the decomposition space to be \emph{locally discrete}, that is the fibres of $s_0$ and $d_1$ are discrete groupoids, then the comultiplication formula will be free of denominators, see \cite{GKT:combinatorics}.

For any locally finite decomposition space $X$, we can take the cardinality of the linear functors $\Delta$ and $\varepsilon$ to obtain a coalgebra structure $(\Q_{\pi_0{X_1}}, |\Delta|, |\varepsilon|)$, called the \emph{numerical incidence coalgebra} of $X$.

Under the completeness condition ($s_0$ is a monomorphism) and the locally finite length condition (each edge has a finite length), there is a general Möbius inversion principle, see \cite{GKT2}.

\paragraph{Culf maps}
A map $f\colon X \to Y$ of simplicial spaces is \emph{cartesian} on an arrow \mbox{$[n] \to [k]$} in $\simplexcat$ if the naturality square for $f$ with respect to this arrow is a pullback.

A simplicial map $f\colon X \to Y$ is \emph{culf} if it is cartesian on all active maps.
The culf functors induce coalgebra homomorphisms between the incidence coalgebras \cite[Lemma 8.2]{GKT1}. 

\begin{prop}[{\cite[Lemma 4.6]{GKT1}}]\label{culfoverds}
    If $X$ is a decomposition space and $f\colon Y \to X$ is a culf map, then also $Y$ is a decomposition space.
\end{prop}

\begin{prop}[{\cite{GKT:restrict}}]\label{propCULF}
    If $X$ is a locally discrete (resp. locally finite length) decomposition space and $f\colon Y \to X$ is a culf map, then also $Y$ is a locally discrete (resp. locally finite length) decomposition space.
    This is also the case for locally finite, but we must check moreover that $Y_1$ is locally finite.
\end{prop}

\paragraph{Monoidal decomposition spaces, bialgebras}
A \emph{monoidal structure}~\cite[\S 9]{GKT1} on a decomposition space $X$ is given by the data of simplicial maps $\eta\colon 1 \to X$ and $\otimes\colon X \times X \to X$ required to be culf, and satisfying the standard associative and unital laws of monoidal structures. 
This monoidal structure induces a monoid structure on the incidence coalgebra $\Grpd_{/X_1}$ and altogether a bialgebra structure.
In combinatorics, the monoidal structure is often given by disjoint union.


\section{Hereditary species}\label{sec:hersp}

Let $\B$ denote the category of finite sets and bijections, 
let $\I$ denote the category of finite sets and injective maps, and let $\Sur$ denote the category of finite sets and surjective maps.
We denote by $\Sp$ the category of finite sets and partially defined surjections. A partially defined surjection $V \to W$ consists of a subset $U \subset V$ and a surjection $U \to W$. More formally, the arrows in $\Sp$ are given by equivalence classes of spans
\begin{center}
    \begin{tikzcd}
          & U \ar[dl, tail, "i"'] \ar[dr, two heads, "p"] &  \\ 
        V & & W
    \end{tikzcd}    
\end{center}
where $i$ is injective and $p$ is surjective, and where two such spans are equivalent if they are isomorphic as spans. Partially defined surjections are composed by pullback composition of spans (in the category of sets). This is meaningful since both injections and surjections are stable under pullbacks in the category of sets. Note that the empty set is included here. Note also that the category $\Sp$ contains the category $\Sur$ as a subcategory (the spans in which the injection leg is an identity map) and also contains the category $\I\op$ as a subcategory (the spans in which the surjection leg is an identity map).

\paragraph{Species}

Recall that a \emph{species}~\cite{Joyal:species} is a functor $F\colon\B\to\Set$ , $V \mapsto F[V]$. An element of $F[V]$ is called an $F$-structure on the finite set $V$. A \emph{restriction species}~\cite{Schmitt} is a functor $R\colon\I\op\to\Set$. An $R$-structure on a set $V$ thus restricts to any subset $U \subset V$. Schmitt~\cite{Schmitt} further defines a \emph{hereditary species} to be a functor $H\colon \Sp\to \Set$.
An element $G\in H[V]$ is called a \emph{$H$-structure} on the set $V$. A hereditary species is thus covariantly functorial (not only in bijections but also) in surjections, and also contravariantly functorial in injections (that is, is a restriction species). This means that a $H$-structure on a set $V$ induces also a $H$-structure on any quotient set and on any subset. Furthermore, these functorialities are compatible in the sense that for any pullback square
\begin{center}
    \begin{tikzcd}
    U \ar[d, "p"', twoheadrightarrow]
    & U' \ar[l, "i"', rightarrowtail] \ar[d, "p'", twoheadrightarrow] \dlpbk\\
    V & V' \ar[l, "j", rightarrowtail]
    \end{tikzcd}
\end{center}
we have
\[
H[p']\circ H[i] = H[j] \circ H[p].
\]
This `Beck-Chevalley' law is a consequence of the fact that $H$
must respect the composition of spans.

If $\pi$ is a partition of $V$, and $\rho_{V,\pi}\colon V \to \pi$ is the canonical surjection, the \emph{quotient} $G/\pi$ is the $H$-structure on the set $\pi$ defined by 
\[G/\pi = H[\rho_{V,\pi}](G).\]
The \emph{restriction} $G|\pi$ is defined to be the family 
\[G|\pi = \{ G|B \}_{B \in \pi}.\]
A morphism of hereditary species is a natural transformation of functors. We denote by $\kat{HSp}$ the category of hereditary species and natural transformations.

\begin{ex}\label{graphs}
   For a graph $G$ with vertex set $V$, and $\pi$ a partition of $V$, we define $G|\pi$ to be the family of graphs whose vertex sets are blocks of $\pi$ and with an edge between two elements of the same block if there is an edge in $G$ with both incident vertices in the block.
We define $G/\pi$ to be the graph with vertex set $\pi$ and with an edge between two vertices if there is a edge in $G$ between the corresponding blocks.
\end{ex}

Suppose that $\tau$ is a finer partition than $\sigma$ (denoted $\tau \le \sigma$), that is each block of $\sigma$ is a union of blocks of $\tau$. We denote $\sigma/\tau$ the partition of the set $\tau$ induced by $\sigma$.

The following proposition is a consequence of the functorialities in surjections and injections, and the Beck-Chevalley law.

\begin{prop}[{\cite[Proposition 4.1]{Schmitt}}]\label{schmittidentities}
If $\tau$ and $\sigma$ are partitions of $V$ such that $\tau \le \sigma$, 
then the following identities hold:
    \begin{align*} 
        [(G|\sigma)|\tau] &= [G|\tau], \\
        [(G/\tau)|(\sigma/\tau)] &= [(G|\sigma)/\tau], \\
        [(G/\tau)/(\sigma/\tau)] &= [G/\sigma],
    \end{align*}
    where $[G]$ is the isomorphism class of $G$.
\end{prop}

\paragraph{Coalgebra}\label{hercoalgebra}
To any hereditary species, Schmitt associates a coalgebra $B$ on the vector space spanned by all isomorphism classes of families of \emph{non-empty} $H$-structures. The comultiplication of a $H$-structure $G$ on the set $V$ is defined by:
 \[ \Delta(G) = \sum_{\sigma \in \Pi(V)} G|\sigma \otimes G/\sigma.\]
The counit is defined by
 \[
 \varepsilon(G) = 
 \begin{cases}
     1 &\text{ if every member in the family is a singleton,}\\
     0 &\text{ otherwise.}
 \end{cases}
 \]

Note the importance of disallowing empty structures.
With the empty $H$-structure we would have $\Delta(\emptyset) = ( \,) \otimes \emptyset$, and the comultiplication would not be counital on the right. While not counital on the right, we shall see later that it is still a left comodule over $B$.

\paragraph{Hereditary species of non-empty sets}
We consider the hereditary species of non-empty finite sets. 
The comultiplication of a finite set $V$ is defined by summing over all partitions of $V$ and putting on the left the family of blocks of the partition and on the right the set whose elements are blocks of the partition:
 \[ 
    \Delta(V) = \sum_{\pi \in \Pi(V)} (V_1, \dots, V_k) \otimes \pi.
\]

\section{A decomposition space of surjections}
\label{sec:dssurj}

We work with groupoids instead of sets, to take into account  symmetries (see \cite{BD} and \cite{GKT:combinatorics})
and define a \emph{hereditary species} to be a functor 
\[H\colon \Sp \to \Grpd.\]
In particular, a hereditary species is (covariantly) functorial in surjections and contravariantly functorial in injections, and these two functorialities interact via the Beck-Chevalley rule.

\paragraph{Partitions and surjections}
The groupoid of partitions (whose objects are sets with a partition into blocks, and arrows are bijections between sets preserving blocks) is equivalent to the groupoid of surjections (arrows are pairs of compatible bijections).
More precisely, a partition of a set $V$ can be given by a surjection $V \surj P$: a block of the partition of $V$ is the preimage of an element of $P$.
Refinement of partitions is rendered conveniently in terms of composition of surjections. Precisely, the poset of partitions
of $V$ under refinement is equivalent to the coslice $\Sur_{V/}$: to say that $\rho \le \pi$ is precisely to say that there is a commutative triangle of the corresponding surjections 
\begin{center}
    \begin{tikzcd}[column sep=small]
       & V \ar[dr, two heads,""] 
           \ar[dl, two heads,""']  & \\
        P \ar[rr, two heads, ""] & & Q.
    \end{tikzcd}
\end{center}
More generally, given $n$ composable surjections from $V$, we get $n$ partitions of $V$.
A pair of composable surjections $V \stackrel{f}{\surj} P \stackrel{g}{\surj} Q$ induces surjective maps between the fibres of $gf$ and $g$ over each element of $Q$.

\paragraph{Fat nerve of the category of finite sets and surjections}

The \emph{fat nerve}~\cite{GKT1}\index{nerve!fat} of $\Sur$ is the simplicial groupoid
\begin{align*}
    \fatnerve \Sur\colon \simplexcat\op &\to \Grpd \\
                         [n] & \mapsto \Map([n], \Sur).
\end{align*}
Explicitly, $(\fatnerve \Sur)_0$ is the groupoid of finite sets and bijections, $(\fatnerve \Sur)_1$ is the groupoid whose objects are surjections and maps consist of a bijection between the sources, and a compatible bijection between the targets. The objects of the groupoid $(\fatnerve \Sur)_n$ are $n$ composable surjections, and maps are $(n{+}1)$-uplets of compatible bijections.
The inner face maps are given by composition, the outer face maps by forgetting the first, or the last set in the chain.
The degeneracy maps are given by inserting identity maps.

The skeleton $\Sur_{\text{ord}}$ of $\Sur$ consisting of ordinal numbers and surjections is a full subcategory of $\Sur$, and $\fatnerve \Sur_{\text{ord}} \eq  \fatnerve \Sur$.

\paragraph{Symmetric monoidal category monad}
The symmetric monoidal category monad $\Smonad\colon \Grpd \to \Grpd$~\cite[\S 2.5]{KW} is the monad represented by the polynomial
  \[
  1\leftarrow \B' \to \B \to 1
  \]
  where $\B$ is the groupoid of finite ordinals and bijections (not required to be monotone), and $\B'$ is the groupoid of finite pointed ordinals and basepoint-preserving bijections. 
It sends a groupoid $X$ to
  \[
  \int^{n\in \B} \Map(\B'_n,X) \simeq \int^{n\in \B} X^{\underline n},
  \]
  where $\underline n$ denotes the fibre over $n$, and the integral sign is a homotopy sum~\cite{GKT:HLA}:
  \[
      \int^{k} X = \sum_k \frac{X}{\Aut k}.
  \]
Given a groupoid $X$, on objects $\Smonad X$ is the groupoid whose objects are finite lists of objects of $X$, and a morphism $ (a_1, \dots, a_n) \to (b_1, \dots , b_m)$ consists of a bijection $\underline{n} \to \underline{m}$ and morphisms $a_i \to b_{\sigma(i)}$ in $X$.
The algebras over $\Smonad$ are symmetric monoidal categories. The unit sends an element $l$ to the list with one element $(l)$, and the multiplication concatenates the lists.

\paragraph{Simplicial groupoid of surjections}

Consider the hereditary species of non-empty sets. Associated to it we construct a simplicial groupoid $\ds S$. Later this simplicial groupoid $\ds S$ will be the base ingredient in the construction of a simplicial groupoid $\ds H$ associated to each hereditary species $H$. The simplicial groupoid $\ds H$ will be a monoidal decomposition space, and therefore define a bialgebra, which will be shown to be the Schmitt bialgebra construction. The basis elements of this bialgebra (i.e.~the objects of the groupoid $\ds H_1$), will be \emph{families} of non-empty $H$-structures, not individual non-empty $H$-structures. Similarly for $\ds S$, the basis elements, the objects of $\ds S_1$, will be \emph{families} of non-empty finite sets, not just individual non-empty finite sets.
Including families rather than just individual structures is necessary in order to have a well-defined comultiplication, since the left-hand tensor factor will be a monomial rather  than a linear factor, as explained in the introduction.  At the same time, working with families gives immediately the algebra  structure (which is free).  Nevertheless, it will be technically important to consider also individual structures, which we  regard as \emph{connected} families.

Accordingly, we first describe a 
simplicial-groupoid-with-missing-top-face-maps, which we call 
$\ds C$ for `connected'.
We first consider the groupoid $\ds{C}_j$ of $(j{-}1)$ composable surjections between non-empty finite sets. The objects of $\ds{C}_2$ are surjections, the objects of $\ds{C}_1$ are non-empty finite sets, and $\ds{C}_0$ is the terminal groupoid, that is equivalent to a point.
Face maps are given by:
\begin{itemize}
    \item $d_0$ forgets the first set in the chain of surjections;
    \item $d_i\colon \ds{C}_j \to \ds{C}_{j-1}$ compose the $i$th and $(i{+}1)$st surjection, for $0 < i < j-1$;
    \item $d_{j-1}$ forgets the last set in the chain of surjections.
\end{itemize}
The degeneracy maps $s_i\colon \ds{C}_j \to \ds{C}_{j+1}$, for $0 \le i \le j-1$ are given by inserting an identity arrow at object number $i$. 
The degeneracy map $s_\top\colon \ds{C}_j \to \ds{C}_{j+1}$ is given by appending with the map whose target is the terminal set $1$.

\begin{rmk}
    The map to $1$ is a surjection since the sets were required non-empty. With possibly empty sets, it would not be possible to define the top degeneracy map.
\end{rmk}

To obtain top face maps, it is necessary to introduce families: the top face map of a surjections chain must be the family of surjections chain shorter by one, obtained as the fibre over each element in the last set. 
We define $\ds S$ to be the symmetric monoidal category monad $\Smonad$ applied to $\ds C$.  All the face maps (except the  missing top ones) and all the degeneracy maps are just $\Smonad$ applied to the face and degeneracy maps of $\ds C$.
The missing top face map $d_\top$ is now given by fibres: given $(j{-}1)$ composable surjections, for each element $k$ of the last target set, we can form the fibres over $k$ of the different source sets.  It also induces surjective maps between these different fibres.
We end up with a family (indexed by elements of the last target set) of $(j{-}2)$ composable surjections between the fibres.
Note that the fibres are non-empty since we only consider surjections, not arbitrary maps.

\begin{rmk}
    Simplicially, the multi aspect is localised to the top face maps. This is already a feature well known from operads~\cite{KW}: the `domain' (given by the simplicial map $d_1$) of a single operation of an operad is not a single object but a family of objects. In the present situation, beyond operads, a new feature is that the top face maps do not satisfy the simplicial identities on the nose, only up to coherent homotopy, making altogether the simplicial groupoids pseudosimplicial (as already allowed for in the homotopy setting in which the theory of decomposition spaces is staged).  This is caused by symmetries acting on blocks of partitions, and hence on factors in monomials, and seems to be an unavoidable nuisance, except in fully rigid situations such as $L$-species with monotone surjections. 
\end{rmk}

\begin{prop}
    The groupoids $\ds{S}_j$ and the degeneracy and face maps given above form a pseudosimplicial groupoid $\ds{S}$.
\end{prop}

\begin{proof}
The only pseudosimplicial identity is
$d_\top \circ d_\top \eq d_\top \circ d_{\top-1}$.
The other simplicial identities are strict and straightforward to check. The ones involving top face and top degeneracy maps are:
\begin{align*}
    d_i \circ d_\top &= d_\top \circ d_{i} &
 	s_i \circ d_\top &= d_\top \circ s_{i} &
	d_\top \circ s_\top &= \id \\
	d_\top \circ d_\bot &= d_\bot \circ d_\top &
	d_i \circ s_\top &= s_\top \circ d_i &
	s_\top \circ s_i &= s_i \circ s_\top 
  \end{align*} 

It remains to verify the functor is pseudosimplicial for the top face map $d_\top$.
Given two composable surjections $V \stackrel{f}\surj P \stackrel{g}\surj Q$, we can consider the family $\{V_p\}_{p \in P}$ of fibres of $f$ over elements of $P$, and we can also consider the family of families $\left\{  \{ V_p \}_{p \in P_q} \right\}_{q\in Q}$.
There is a canonical isomorphism $\left\{  \{ V_p \}_{p \in P_q} \right\}_{q\in Q} \to \{V_p\}_{p \in P}$.
We want to show the following square is commutative:
\begin{center}
    \begin{tikzcd}
        ((V_k)_l)_m \ar[r, ""] 
           \ar[d, ""'] & (V_k)_l \ar[d, ""]\\
        (V_k)_p \ar[r, ""'] & V_p.
    \end{tikzcd}
\end{center}
The isomorphisms are compatible with the injections of fibres into $V$, by the universal property of pullback:
given $p \in P$ such that $g(p)=q$, the following square is a pullback by definition:
\begin{center}
    \begin{tikzcd}
        V_p \ar[r, ""] 
           \ar[d, ""'] & V \ar[d, ""]\\
        1 \ar[r, "\name{p}"'] & P.
    \end{tikzcd}
\end{center}
We also have the following commutative diagram, given by pullbacks
\begin{center}
    \begin{tikzcd}
      (V_q)_p \ar[r]\ar[d]\ & V_q \ar[r, ""] 
         \ar[d, ""'] & V \ar[d, ""]\\
     1 \ar[r, "\name{p}"'] &  P_q \ar[r, ""']\ar[d]\ & P \ar[d]\\
     & 1 \ar[r, "\name{q}"'] & Q.
    \end{tikzcd}
\end{center}
Hence by the universal property of pullbacks, there exists a unique isomorphism $(V_q)_p \to V_p$ such that the following triangle commutes:
\begin{center}
    \begin{tikzcd}
        (V_q)_p \ar[dr, ] 
           \ar[d, "\eq"'] & \\
         V_p \ar[r, ""'] & V.
    \end{tikzcd}
\end{center}
In a similar way, we obtain the three other isomorphisms between the fibres compatible with the injections into $V$.
Since all four isomorphisms in the square are compatible with the injections into $V$, the commutativity of the square is ensured by the fact that each $V_i \to V$ is a monomorphism.
\end{proof}

Considering only non-empty ordinals $\underline{n}$ instead of all non-empty finite sets, we get an equivalent pseudo\-simplicial groupoid, which we denote $\ds{S}_{\text{ord}}$.
The objects of $(\ds{S}_{\text{ord}})_1$ are lists of non-empty  ordinals, and a map between two lists $(\underline{n}_1, \dots,\underline{n}_p)$ and $(\underline{m}_1, \dots,\underline{m}_{p'})$ consists of a bijection $\sigma\colon p \to p'$ and, for all $i\in p$, a map \mbox{$\rho_i\colon \underline{n}_i \to \underline{m}_{\sigma(i)}$}.

\begin{prop}\label{equivS}
We have a natural equivalence
    $\fatnerve \Sur_\text{ord} \eq \ds{S}_{\text{ord}}$.
\end{prop}

\begin{proof}
    The functor $\fatnerve \Sur_\text{ord} \to \ds{S}_{\text{ord}}$ sends a surjection $\underline{n} \surj \underline{p}$ to the list of fibres $(\underline{n}_1, \dots,\underline{n}_p)$, and a map $(\underline{n} \xra{\rho} \underline{n}', \underline{p} \xra{\sigma} \underline{p}')$ between two surjections $\underline{n} \surj \underline{p}$ and $\underline{n}' \surj \underline{p}'$ to the map
    $(\underline{n}_1, \dots,\underline{n}_p) \xra{(\sigma, \rho_1, \dots, \rho_p)} (\underline{n}'_1, \dots,\underline{n}'_{p'})$. Note that the surjection $\emptyset \surj \emptyset$ gives the empty list, which is
allowed.
    Similarly, it sends a family of $n$ composable surjections to the list of $(n{-}1)$ composable surjections given by the fibres. The maps are given in a similar way.
    
    There is also a functor $\ds{S}_{\text{ord}} \to \fatnerve \Sur_\text{ord}$. 
    Given a list of non-empty ordinals $(\underline{n}_1, \dots,\underline{n}_p)$, we obtain a surjection $\sum_{i \in \underline{p}} \underline{n}_i \surj \underline{p}$ from the  disjoint union of the elements of the list to the indexing set. (This map is a surjection because all the $n_i$ are non-empty.)
    Given a map between two lists $(\underline{n}_1, \dots,\underline{n}_p)$ and $(\underline{m}_1, \dots,\underline{m}_{p'})$, we obtain a map $\sum_{i \in \underline{p}} \underline{n}_i \to \sum_{i \in \underline{p}'} \underline{m}_i$ by sending $\underline{n}_i$ to $\underline{m}_{\sigma(i)}$.
    Similarly, given a list of $(j{-}1)$ composable surjections, we obtain $j$ composable surjections by disjoint union, and the last one is given as before, using the target sets of the surjections.

It is easy to check they form an equivalence since the disjoint union of fibres of a surjection is isomorphic to the source set of this surjection.
\end{proof}

\begin{prop}\label{SSSegal}
     The pseudosimplicial groupoid $\ds{S}$ is Segal, and hence a decomposition space.
\end{prop}

\begin{proof}
It follows from the equivalence $\fatnerve \Sur_\text{ord} \eq \ds{S}_{\text{ord}} \eq \ds{S}$ since the fat nerve of a  small category is always Segal~\cite{GKT1}.
\end{proof}

\begin{prop}\label{Xfinite}
    The Segal space $\ds{S}$ is complete, locally finite, locally discrete, and of locally finite length.
\end{prop}

\begin{proof}
    The map $s_0$ is a monomorphism because the fibre is empty if the first surjection is not the identity, and is contractible else.
  The groupoid $\ds{S}_1$ is locally finite, because elements of the families are non-empty finite sets, and each finite set has only a finite number of automorphisms.
We have seen $s_0$ is finite and discrete, the map $d_1$ is also finite and discrete: the fibre of $d_1$ over $(n{-}1)$ composable surjections $f \in \ds{S}_n$ is either empty or the finite discrete groupoid of $n$ composable surjections where the first one is the identity and the other surjections are given by $f$.
Finally, $\ds{S}$ is of locally finite length (every edge $f$ has finite length): the degenerate simplices are families where one of the surjections is an identity, or the last set is a singleton. The fibre of $f$ has no nondegenerate simplices for $n$ greater than the total number of elements of the source sets of the family.
\end{proof}

\begin{rmk}\label{notSegal}
The decomposition space $\ds{H}$ is not usually a Segal space. The base case of the Segal condition stipulates
that the square
\begin{center}
    \begin{tikzcd}
      \ds{H}_2 \ar[r, "d_2"] \ar[d, "d_0"']
          & \ds{H}_1 \ar[d, "d_0"] \\
      \ds{H}_1 \ar[r, "d_1"'] & \ds{H}_0
    \end{tikzcd}
\end{center}
is a pullback. This would mean that one should be able to reconstruct a $H$-structure on a surjection $V \surj P$ by knowing it on $P$ and on the fibres $V_1,\dots,V_p$.
In other words, one could substitute the $H$-structures on $V_1,\ldots, V_p$ into the elements of $P$ of another $H$-structure, as if those elements were `input slots' of the operation of an operad.

Consider for example the case of simple graphs (Example~\ref{graphs}). Given $p$ simple graphs with vertex sets $V_1,\dots, V_p$ and another graph with $p$ vertices, there is no canonical prescription for substituting the $p$ graphs into those vertices.
\end{rmk}

Since in every degree, the groupoid is given by applying $\Smonad$, the decomposition space $\ds{S}$ is automatically a monoidal decomposition space.
The associated incidence bialgebra has the property that the comultiplication applied to a connected element gives a monomial in the left-hand tensor factor and a connected element (linear  factor) in the right-hand tensor factor. That's the immediate conclusion of the fact that $d_2$ requires $\Smonad$ whereas $d_0$ does not.
 This observation can be formalised at the simplicial level by the following result.
 
Recall from \cite[Proposition 2.1.1]{Ca:bicomodules} (see also \cite{Walde} and \cite{Young}) that the decomposition space analog of (left) comodule is given by a simplicial map $f\colon C \to X$ between two simplicial groupoids such that $C$ is Segal, $X$ is a decomposition space and the map $f\colon C \to X$ is culf.
Then the span
\[
    C_{0}  \xleftarrow{d_1} C_{1} 
    \xrightarrow{(f_{1},d_0)} X_{1} \times C_{0}
\]
induces on the slice category $\Grpd_{/C_0}$ the structure of a left $\Grpd_{/X_1}$-comodule.

\begin{lem}\label{CcomodS}
  The slice category $\Grpd_{/\ds{C}_1}$ is a left comodule over $\Grpd_{/\ds{S}_1}$.
\end{lem}

\begin{proof}
Note that $\ds{C}$ is lacking top face maps, but $\Dec_\top \ds{C}$ is a genuine simplicial groupoid as required by the notion of comodule.
Since $\Dec_\top(\ds{C})$ is a Segal space (the décalage of a decomposition space is always a Segal space~\cite[Proposition 4.9]{GKT1}), and $\ds{S}$ is a decomposition space, we just need to exhibit a culf map $\Dec_\top(\ds{C}) \to \ds{S}$, which is given by $d_\top$.
  Note that it is essentially the unit for the monad $\Smonad$, therefore it is cartesian, and in particular culf.
\end{proof}

In fact, we can consider all (possibly empty) finite sets, and still obtain a comodule structure. (The surjection $\emptyset \to \emptyset$ will be sent to the empty family.)

\begin{lem}\label{NcomodS}
  The slice category $\Grpd_{/\fatnerve{\Sur}_0}$ is a left comodule over $\Grpd_{/\ds{S}_1}$.
\end{lem}

\begin{proof}
The fat nerve of a category is always a Segal space~\cite[\S 2.14]{GKT1}.
We need to exhibit a culf map $\fatnerve{\Sur} \to \ds{S}$, which is given by sending a surjection to the family of fibres (which are non-empty), as the map $d_\top$ of the decomposition space $\ds{S}$.
\end{proof}

\section{Hereditary species and decomposition spaces}\label{sec:hersp&ds}

We can now add a hereditary structure on the source of each member of the family. Given a hereditary species $H$, define $\ds{H}_1$ to be the groupoid of families of non-empty $H$-structures. 
More formally, $\ds{H}_1$ is defined as families of the Grothendieck construction of the underlying ordinary species $H\colon \B_+ \to \Grpd$, where $\B_+$ denote the category of non-empty finite sets and bijections.
The groupoid $\ds{H}_n$ is given by the pullback
\begin{center}
    \begin{tikzcd}
        \ds{H}_n \ar[r, ""] 
           \ar[d, ""'] 
           \drpbk
           & \ds{H}_1 \ar[d, ""]\\
        \ds{S}_n \ar[r, "\text{source}"'] & \ds{S}_1.
    \end{tikzcd}
\end{center}
Inner face maps and degeneracy maps are defined by pullback.
For example, in the following diagram, the right-hand square  and the rectangle are pullbacks by definition
\begin{center}
    \begin{tikzcd}
        \ds{H}_{3} \ar[r, ""] 
           \ar[d, ""']  
           & \ds{H}_{2} \ar[d, ""] \ar[r]  \drpbk & \ds{H}_1 \ar[d]\\
        \ds{S}_3 \ar[r, "d_2"'] 
           & \ds{S}_2 \ar[r, "d_1"'] & \ds{S}_1.
    \end{tikzcd}
\end{center}
Thus the left-hand square is a pullback and it induces a map $d_2\colon \ds{H}_3 \to \ds{H}_2$.
For the bottom face maps we use that $H$ is functorial in surjections: given an object in $\ds H_n$, that is a chain of surjections $V_1 \surj V_2 \surj \dots \surj V_n$ with a $H$-structure on $V_1$, we get a $H$-structure on $V_2$ by functoriality, and thus an object $V_2 \surj \dots \surj V_n$ of $\ds{H}_{n-1}$.
 For the top face maps, we use (contravariant) functoriality in injections, this is restriction to each of the fibres, and produces thus a family, even if the input is a single chain of surjections.

\begin{prop}
    The groupoids $\ds{H}_j$ form a simplicial groupoid $\ds{H}$.
\end{prop}

\begin{proof}
Checking the simplicial identities requires precisely the three functorialities of the notion of hereditary species (covariant in surjections, contravariant in injections, and Beck--Chevalley condition).
For example, the maps $d_\top$ and $d_\bot$ are given by the following pullbacks
\begin{center}
    \begin{tikzcd}
        \ds{H}_{n+1} \ar[r, "d_\top"] 
           \ar[d, ""']  \drpbk
           & \ds{H}_{n} \drpbk \ar[d, "f"] \ar[r, "d_\bot"] & \ds{H}_{n-1} \ar[d, "g"]\\
        \ds{S}_{n+1} \ar[r, "d_\top"'] 
           & \ds{S}_n \ar[r, "d_\bot"'] & \ds{S}_{n-1}\\
    \end{tikzcd}
\end{center}
and by Beck--Chevalley: 
\begin{align*}
    (d_\bot)_!f^*d_\top & \eq g^*((d_\bot)_!d_\top) \\
    (d_\bot)! d_\top    & \eq g^*(d_\bot d_\top) \\
     d_\bot d_\top      & \eq g^*(d_\top d_\bot) \\
     d_\bot d_\top      & \eq d_\top d_\bot.
\end{align*}
\end{proof}

\begin{prop}\label{Hds}
    The simplicial groupoid $\ds{H}$ is a monoidal decomposition space.
\end{prop}

\begin{proof}
This follows from Proposition~\ref{culfoverds} because there is a culf map to the decomposition space $\ds{S}$ (in fact even a Segal space) by construction. The monoidal structure is obtained by concatenation of families.
\end{proof}

\begin{lem}\label{H1locfinite}
    The groupoid $\ds{H}_1$ is locally finite.
\end{lem}

\begin{proof}
    The objects of the groupoid $\ds{H}_1$ are families of non-empty $H$-structures, and elements of the families are non-empty finite sets. Since a finite set has only a finite number of automorphism, the automorphisms groups are finite.
\end{proof}

\begin{prop}\label{Hfinite}
    The decomposition space $\ds{H}$ is complete, locally finite, locally discrete, and of locally finite length.
\end{prop}

\begin{proof}
    It follows from Lemmas~\ref{propCULF} and \ref{H1locfinite}, and Proposition~\ref{Xfinite} since there is a functor $\ds{H} \to \ds{S}$ which is culf by construction.
\end{proof}

\begin{prop}\label{Schmittcoincide}
    The incidence bialgebra $B$ obtained by taking homotopy cardinality of $\Grpd_{/\ds{H}_1}$ coincides with the Schmitt bialgebra associated to a hereditary species $H$.
\end{prop}

\begin{proof}
The comultiplication is given by pullback along $d_1$ followed by composition with $(d_2,d_0)$.
Given a $H$-structure $G$ as a morphism $1 \xra{\name{G}} \ds{H}_1$, the pullback along $d_1$ is the groupoid $(\ds{H}_2)_G$ of families of surjections with the source sets given by $G$ (thus having $H$-structure).
Composing with $(d_2,d_0)$ amounts to returning for each surjection $V \surj P$ the family of fibres over the element of $P$ (having $H$-structure given by restriction), and the target set $P$ (having $H$-structure given by quotient).
We can then take cardinality since $\ds{H}$ is locally finite. Since $\ds H$ is locally discrete, $(\ds{H}_2)_G$ is discrete and the homotopy cardinality counts isomorphisms classes, giving Schmitt's comultiplication.
\end{proof}

\paragraph{Comodule structure}

Given a hereditary species $H$, define $\ds{M}_0$ to be the groupoid of (possibly empty) $H$-structures. 
More formally, $\ds{M}_0$ is defined as the Grothendieck construction of the underlying ordinary species $H\colon \B \to \Grpd$.
The groupoid $\ds{M}_n$ is given by the pullback
\begin{center}
    \begin{tikzcd}
        \ds{M}_n \ar[r, ""] 
           \ar[d, ""'] 
           \drpbk
           & \ds{M}_0 \ar[d, ""]\\
        \fatnerve \Sur_n \ar[r, "\text{source}"'] & \fatnerve\Sur_0.
    \end{tikzcd}
\end{center}
The objects of the groupoid $\ds{M}_1$ are surjections with a $H$-structure on the source; the objects of $\ds{M}_n$ are composable surjections with a $H$-structure on the source.
In the same way as for $\ds{H}$, the face and degeneracy maps are defined by pullback.

\begin{lem}
  The groupoids $\ds{M}_n$ form a Segal space $\ds{M}$. 
\end{lem}

\begin{proof}
    It is easy to check that they assemble into a simplicial groupoid, using the three functorialities of the notion of hereditary species.
    The simplicial groupoid is equivalent to the fat nerve of the Grothendieck construction of the underlying species $H\colon \Sur \to \Grpd$.
    Whereas $\ds{H}$ is defined as a chain of surjections with a $H$-structure on the source, in the Grothendieck construction, an $n$-simplex is a chain with an $H$-structure on \emph{each} set, and with specified isomorphisms with the $H$-structures pushed forward along the surjections. The presence of these specified isomorphisms readily shows that the two simplicial groupoids are level-wise equivalent.
    Thus $\ds{M}$ is Segal since the fat nerve of a  small category is always Segal~\cite{GKT1}.
\end{proof}

\begin{lem}\label{Hcomodule}
  The slice category $\Grpd_{/\ds{M}_0}$ is a left comodule over $\Grpd_{/\ds{H}_1}$.
\end{lem}

\begin{proof}
   The culf map $f\colon \ds{M} \to \ds{H}$ is given by taking fibres, as the map $d_\top$ of $\ds{H}$.
 Thus the span
    \[
    \ds{M}_{0}  \xleftarrow{d_1} \ds{M}_{1} 
    \xrightarrow{(f_{1},d_0)} \ds{H}_1 \times \ds{M}_{0}
\]
induces on $\Grpd_{/\ds{M}_0}$ the structure of a comodule over $\Grpd_{/\ds{H}_1}$.
\end{proof}

\section{The incidence comodule bialgebra of a hereditary species}\label{sec:comodulebialgebra}

We have constructed, via a monoidal decomposition space, the incidence bialgebra $B$ of a hereditary species $H$. It is the vector space spanned by all families of non-empty $H$-structures. But every hereditary species is in particular a restriction species, by precomposition with the inclusion $\I\op \to \Sp$.
Therefore there is another coalgebra, linearly spanned by the (possibly empty) $H$-structures themselves. The free algebra on this coalgebra is therefore the bialgebra $A$ linearly spanned by the families of $H$-structures.
So now we have two different bialgebra structures on closely related vector spaces, and the two share the same multiplication.
The main result of this section relates these two structures.
\begin{prop}\label{prop:comodulebialg}
  The hereditary-species bialgebra $B$ coacts on the restriction-species bialgebra $A$, so as to make $A$ a left comodule bialgebra over $B$.
\end{prop}

The proof is a nice illustration of the objective method: after  unpacking the definitions, the proof consists in computing a  few pullbacks. Let us first recall some definitions and set notation.

\paragraph{Hereditary species and decomposition spaces}
Given a hereditary species $H\colon \Sp \to \Grpd$, we get a decomposition space $\ds{H}$ where an $n$-simplex is a family of $(n{-}1)$ composable surjections between non-empty finite sets, with a $H$-structure on each source set.
The comultiplication
$\Grpd_{/\ds{H}_1} \to  \Grpd_{/\ds{H}_1 \times \ds{H}_1}$ 
is given by the span
\[
    \ds{H}_1 \xleftarrow{d_1} \ds{H}_2 \xrightarrow{(d_2, d_0)} \ds{H}_1 \times \ds{H}_1,
\]
where $d_1$ returns the family of source sets, $d_0$ returns the family of target sets, and $d_2$ returns the family of families of fibres over each element of the target sets.
Let $B$ denote the homotopy cardinality, i.e.~the numerical incidence bialgebra of $\ds{H}$.

\paragraph{Restriction species and decomposition spaces}
Since $\I\op$ is a subcategory of $\Sp$, every hereditary  species $H$ induces a restriction species $\colon \I\op \to \Grpd$. Since the hereditary species $H$ is fixed throughout this section, we denote the underlying restriction species simply by $R$.
Recall from~\cite{GKT:restrict} that every restriction species $R$ induces a decomposition space $\ds{R}$  where an $n$-simplex is an $n$-layered set with an $R$-structure on the underlying set.
The comultiplication 
$\Delta\colon \Grpd_{/\ds{R}_1} \to  \Grpd_{/\ds{R}_1 \times \ds{R}_1}$ 
is given by the span
\[
    \ds{R}_1 \xleftarrow{d_1} \ds{R}_2 \xrightarrow{(d_2,d_0)} \ds{R}_1 \times \ds{R}_1,
\]
where $d_1$ joins the two layers of the $2$-simplex, and $d_2$ and $d_0$ return the first and second layers respectively.
Note that $\ds{R}_1 = \ds{M}_0$ and by Lemma~\ref{Hcomodule} the slice category $\Grpd_{/\ds{R}_1}$ is a left $\Grpd_{/\ds{H}_1}$-comodule.

\paragraph{Comodule bialgebra}
For background on comodule bialgebras, see \cite[\S 3.2]{Abe}.
Let $B$ be a bialgebra.  Recall that a \emph{(left) $B$-comodule  bialgebra} is a bialgebra object in the braided monoidal  category of left $B$-comodules.  
For any coalgebra $B$ we have the category of left $B$-comodules. A left $B$-comodule is a vector space $M$ equipped  with a coaction $\gamma\colon M \to B \tensor M$ satisfying the usual axioms.
So far only the coalgebra structure of $B$ is needed.
The algebra structure of $B$ comes in to provide a (braided) monoidal structure on the category of left $B$-comodules.
It is given as follows. If $M$ and $N$ are left $B$-comodules, then $M\tensor N$ is given a left $B$-comodule structure by the composite map
\[
  M \tensor N \to B \tensor M \tensor B \tensor N 
    \stackrel{\omega}\to B \tensor M \tensor N,
\]
where the map $\omega$ is given by first swapping the two middle  tensor factors, and then using the multiplication of $B$ in the two now adjacent $B$-factors. 
It follows from the bialgebra axioms that this is a valid left $B$-coalgebra structure.
This defines the monoidal structure on the category of left  $B$-comodules.  The unit object for this monoidal structure is the $B$-comodule $\Q$ (with structure map the unit of $B$).
It is easy to check that the underlying braiding of the category of vector spaces lifts to a braiding on this monoidal structure.

We now have a braided monoidal structure on the category of left $B$-comodules, and it makes sense to consider bialgebras in here. For reference, let us recall that a bialgebra in the braided monoidal category of left $B$-comodules is a $B$-comodule $M$ together with structure maps
\begin{xalignat*}{2}
\Delta_M\colon M &\to M \tensor M  & \varepsilon_M\colon M &\to \Q \\
\mu_M\colon M \tensor M &\to M & \eta_M\colon \Q &\to M
\end{xalignat*}
which are all required to be $B$-comodule maps and to satisfy the usual bialgebra axioms.
We shall be concerned in particular with the requirement that
$\Delta$ and $\varepsilon$ be $B$-comodule maps:
\begin{center}
\begin{tikzcd}
M \ar[r, "\Delta_M"] \ar[dd, "\gamma"'] & M \tensor M \ar[d, "\gamma \tensor 
\gamma"] \\
& B \tensor M \tensor B \tensor M \ar[d, "\omega"] \\
B \tensor M \ar[r, "B \tensor \Delta_M"'] & B \tensor M \tensor M
\end{tikzcd}
\qquad
\begin{tikzcd}
M \ar[d, "\gamma"'] \ar[r, "\varepsilon_M"] & \Q \ar[d, "\eta_B"] \\
B \tensor M \ar[r, "B\tensor \varepsilon"'] & B 
\end{tikzcd}
\end{center}




To simplify the proof of Proposition~\ref{prop:comodulebialg}, we shall invoke the following general result.
\begin{lemma}\label{freelemma}
  If $M$ is a comodule coalgebra over $B$, then the free algebra $\SSS M$ is naturally a comodule bialgebra over $B$.
\end{lemma}

\begin{proof}
  If $\gamma\colon M \to B \tensor M$ is the coaction for $M$, then the new coaction
  $\overline \gamma \colon \SSS M \to B \otimes \SSS M$ is given by extending
  multiplicatively, and using the algebra
  structure of $B$:
  \[
  \SSS M \xrightarrow{\SSS(\gamma)} \SSS(B \otimes M) \longrightarrow \SSS B
  \otimes \SSS M \xrightarrow{\mu_B \otimes \Id_{\SSS M}} B \otimes \SSS M .
  \]
  Here the middle map is the oplax-monoidal structure of $\SSS$.  If $\Delta_M\colon M \to M \otimes M$ is the comultiplication of $M$, then the new comultiplication
  $\overline \Delta_M\colon \SSS M \to \SSS M
  \otimes \SSS M$ is given by extending multiplicatively in the usual way:
  \[
  \SSS M \xrightarrow{\SSS(\Delta_M)} \SSS(M \otimes M) \longrightarrow \SSS M
  \otimes \SSS M .
  \]
  It is now straightforward to check that $\overline \Delta_M$ and the new free
  multiplication are $B$-comodule maps for $\overline \gamma$.
\end{proof}

Proposition~\ref{prop:comodulebialg} now follows from the following result, together with
Lemma~\ref{freelemma}. 

\begin{prop}
  Let $A$ be the incidence coalgebra of the ordinary restriction species underlying $H$. Then $A$ is naturally a left $B$-comodule coalgebra (where $B$ is the incidence bialgebra of the hereditary species as in Section~\ref{sec:hersp&ds}).
\end{prop}

\begin{proof} 
The underlying vector space of $A$ is the homotopy cardinality of the comodule of Lemma~\ref{Hcomodule}.
It remains to check that the structure maps $\Delta$ and $\varepsilon$ of the incidence coalgebra of the ordinary restriction species are $B$-comodule maps. We need to check that the two above squares are commutative.

The composition of comultiplications is given by composition of spans. We need to exhibit a commutative diagram as follows, such that the bottom left-hand square, and the top right-hand squares are pullbacks:
\begin{center}
  \begin{tikzcd}[column sep=large]
    \ds{R}_1
      & \ds{R}_2 \ar[l, "d_1"'] 
                              \ar[r, "{(d_2,d_0)}"]
      & \ds{R}_1 \times \ds{R}_1\\
      & & \\
    \ds{M}_1 \ar[uu, "d_1"] 
                          \ar[dd, "{(f, d_0)}"'] 
      & X_3 \ar[l, dashed, "\bar d_2"] 
                              \ar[dd, dashed, "{(g,\bar d_0)}"] 
                              \ar[r, dashed, "{\bar d_3}"'] 
                              \ar[uu, dashed, "\bar d_1"]
                              \ar[uur, phantom, "\urcorner", very near start]
                              \ar[ddl, phantom, "\llcorner", very near start]
      & \ds{M}_1 \times \ds{M}_1 \ar[uu, "d_1 \otimes d_1"'] 
                                          \ar[d, ""]\\
                          & & \ds{H}_1 \times \ds{R}_1 \times \ds{H}_1 \times \ds{R}_1 \ar[d, ""]\\
      \ds{H}_1 \times \ds{R}_1
       &  \ds{H}_1 \times \ds{R}_2 \ar[l, "{\id \otimes d_1 }"]  \ar[r, "\id \otimes {(d_2, d_0)}"']
       & \ds{H}_1 \times \ds{R}_1 \times \ds{R}_1.
  \end{tikzcd}
\end{center}
The objects of $\ds{R}_1$ are $H$-structures.
The groupoid $X_3$ consists of pairs of composable maps $V \surj P \to 2$, such that the first one is a surjection, and with a $H$-structure on $V$.
The map $\overline{d}_0$ sends $(V \surj P \to \underline{2})$ to $(P \to \underline{2})$, the map $\overline{d}_1$ sends it to $(V \to \underline{2})$, the map $\overline{d}_2$ to $(V \surj P)$, the map $\overline{d}_3$ to the pair of surjections between the fibres $(V_1 \surj P_1, V_2 \surj P_2)$, and the map $g$ to the family $\{V_i\}_{i \in P}$ of fibres of the surjection $V \surj P$ over all the elements of $P$.

Recall that objects of $\ds{R}_2$ are maps of sets $V \to \underline{2}$, with a $H$-structure on the source.
Objects of $\ds{H}_1$ are surjections with a $H$-structure on the source.

It is straightforward to verify the four squares are commutative, using the functoriality of $H$ and the Beck-Chevalley rule.
The structure on $\ds{H}_1 \times \ds{R}_1 \times \ds{R}_1$ is obtained as follow. 
On $\ds{H}_1$, the structure is given by restriction on the fibres, and the different paths give equivalent output since $H$ is contravariantly functorial in injections.
For the two $\ds{R}_1$, the structure is given by quotient (functoriality in surjections) then restriction taking the left then down path, or by restriction then quotient taking the top then right path; this gives equivalent output by the Beck-Chevalley rule.

The lower left-hand square is a pullback. Indeed after projecting away $\ds{H}_1$, it is enough to verify, by Lemma~\ref{prismlemma}, that the bottom square of the following diagram is a pullback:
\begin{center}
    \begin{tikzcd}[column sep=large]
        \ds{M}_1 \ar[d, "{(f,d_0)}"] 
           \ar[dd, bend right=50, "d_0"']
            & X_3 \ar[l, "\overline d_2"'] \ar[d, "{(g,\overline{d}_0)}"'] 
            \ar[dd, bend left=50, "\overline{d}_0"]
\\
       \ds{H}_1 \times \ds{R}_1 \ar[d] & \ds{H}_1 \times \ds{R}_2 \ar[d] \ar[l, "\id \times d_1"]  \\  
        \ds{R}_1  & \ds{R}_2. \ar[l, "d_1"] 
    \end{tikzcd}
\end{center} 
The fibre of $\overline d_0$ over a element $P \to \underline{2}$ of $\ds{R}_2$, or the fibre of $\overline d_0$ over the element $d_1(P \to \underline{2}) = P$ consist both of pairs with a surjection onto $P$: $V \surj P$ and the map $P \to \underline{2}$. Thus by Lemma~\ref{fibreslemma}, the bottom square is a pullback

The top right-hand square is also a pullback, using Lemma~\ref{fibreslemma} one more time:
the fibre of $\overline d_1$ over a object $V \to \underline{2}$ of $\ds{R}_2$ consists of pairs of composable maps, where the first one is a surjection from $V$, the second one is a map to $\underline{2}$, such that the composition is $V \to \underline{2}$. 
The fibre of $\overline d_1$ over the object $(d_2,d_0)(V \to \underline{2})$ is a pair of surjections with sources $V_1$ and $V_2$. This is equivalent to the fibre of $\overline d_0$ over $P \to \underline{2}$ since we can take the disjoint union of $V_1 \surj P_1$ and $V_2 \surj P_2$ to get a surjection, and we obtain a map to $\underline{2}$ by sending elements of $P_1$ to $1$, and elements of $P_2$ to $2$.

For the counit condition, it is easy to verify that the following diagram is commutative and the two marked squares are pullbacks
\begin{center}
  \begin{tikzcd}[sep=large]
    \ds{R}_1
          & \ds{R}_0 \ar[l, "s_0"'] 
                     \ar[r, ""]
          & 1\\
    \ds{M}_1 \ar[u, "d_1"] 
        \ar[d, "{(f, d_0)}"'] 
          & \ds{R}_0 \ar[l, dashed, ""] 
              \ar[d, dashed, ""] 
              \ar[r, dashed, ""'] 
              \ar[u, dashed, ""]
              \ar[ur, phantom, "\urcorner", very near start]
                              \dlpbk
          & 1 \ar[u, ""'] 
              \ar[d, ""]\\
    \ds{H}_1 \times \ds{R}_1
          &  \ds{H}_1 \times \ds{R}_0 
                \ar[l, "{\id \otimes s_0}"] 
                \ar[r, ""']
          & \ds{H}_1.
  \end{tikzcd}
\end{center}
Indeed, the bottom left pullback is the groupoid of surjections $V \surj P$ with a $H$-structure on $V$, such that the induced $H$-structure on $P$ is an empty $H$-structure. This implies that both $P$ and $V$ are empty, so it is just any $H$-structure on the empty set, which is $\ds{R}_0$.

\end{proof}

\begin{ex}
  The hereditary species of simple graphs, described in Example~\ref{graphs}, induces a comodule bialgebra. The first comultiplication is given by the hereditary structure as in Section~\ref{hercoalgebra}. The secondary comultiplication is given by plain restriction species structure. By Proposition~\ref{prop:comodulebialg}, this is a comodule bialgebra. It has been studied deeply by Foissy~\cite{Foissy:bialgebragraphs}. It is interesting here to see this example as an instance of a general construction.
\end{ex}


\section{Hereditary species and operadic categories}
\label{sec:operadic}


As we have seen, hereditary species are operad-like without being operads, in the sense that they admit a kind of two-sided bar construction, which is not in general a Segal space. A relationship between operadic categories and decomposition spaces was established recently be Garner, Kock, and Weber~\cite{GKW}. They show that certain unary operadic categories are decomposition spaces. The following construction shows that certain \emph{non-unary} operadic categories are decomposition spaces, namely those that come from hereditary species. 

Since we are going to verify the axioms of operadic category in detail,
we list them here, following the formulation of
\cite{GKW}.

\bigskip

Let $\F$ denote the category whose objects are the sets 
$n = \{1, \dots, n\}$ for $n \in \N$ and whose maps are arbitrary functions. We denote by $1 \in 1$ the unique element in the terminal object. In any category with terminal object $1$, we write 
$\tau_X\colon X \to 1$ for the unique map from an object to the terminal.

Given a function
$\varphi\colon  m \to  n$ in $\F$ and
$i \in  n$, there is a unique monotone injection
\begin{equation}\label{eq:eps}
  \varepsilon_{\varphi, i}\colon \varphi^{-1}(i) \rightarrowtail m
\end{equation}
in $\F$ whose image is $\{\,j \in  m\colon \varphi(j) = i\,\}$;
the object $\varphi^{-1}(i)$ is called the \emph{fibre of $\varphi$ at $i$}. Often the map $\varphi$ is clear from the context, and we write simply 
\[
    \varepsilon_i\colon m_i \rightarrowtail m.
\]
If we are given two maps in $\F$, 
$\ell \stackrel{\psi}\to m \stackrel{\varphi}\to n$, then we denote by $\psi^\varphi_{i}$ the unique map comparing the fibres, given by the universal property of pullback:
\begin{equation}\label{pbkeps}
	\begin{tikzcd}
	  \ell_i \drpbk \ar[d, rightarrowtail, "\varepsilon_{i}"'] 
	  \ar[r, "\psi^\varphi_{i}"] & m_i \drpbk 
	  \ar[d, rightarrowtail, , "\varepsilon_{i}"'] \ar[r] 
	  & \{i\} \ar[d, rightarrowtail]\\ 
	  \ell \ar[r, "\psi"'] & m \ar[r, "\varphi"'] & n
	\end{tikzcd}
\end{equation}
and call it the \emph{fibre map of $\psi$ with respect to $\varphi$ at $i$}. 
We usually omit the $\psi$-decoration.

\paragraph{Operadic categories}\label{def:operadic-category}
  An \emph{operadic category}~\cite{BM} is given by the following data:
\begin{enumerate}[label=(D{\arabic*})]
    \item \label{data:operadic-Q1} a category $\C$ endowed with chosen local terminal objects (i.e.~a chosen terminal object in each connected component);
    \item \label{data:operadic-Q2} a \emph{cardinality functor} $\card{\thg}\colon \C \to \F$;
    \item \label{data:operadic-Q3} for each object $X\in \C$ and each $i \in \card{X}$ a \emph{fibre functor}
      \[
        \fibre_{X,i} : \C/X \to \C
      \]
      whose action on objects and morphisms we denote as follows:
      \begin{align*}
        \cd[@C1.3em]{
          Y \ar[rr]^-{f} && X
        } \qquad &\mapsto \qquad f^{-1}(i)\\
        \cd[@C1em@R-0.7em]{{Z} \ar[rr]^-{g} \ar[dr]_-{fg} & &
          {Y} \ar[dl]^-{f} \\ &
          {X}} \qquad &\mapsto \qquad
        g^f_{i}\colon (fg)^{-1}(i) \to f^{-1}(i)\rlap{ ,}
      \end{align*}
      referring to the object $f^{-1}(i)$ as the \emph{fibre
      of $f$ at $i$}, and the morphism\\
    $g^f_{i}\colon (fg)^{-1}(i) \to f^{-1}(i)$ as the \emph{fibre
      map of $g$ with respect to $f$ at~$i$};
  \end{enumerate}
  all subject to the following axioms, where  in~\ref{axQ:BM-fibres-of-local-fibres}, we write $\varepsilon j$ for the image of $j \in {\card f}^{-1}(i)$ under the map $\varepsilon_{\card f, i}\colon {\card f}^{-1}(i) \to \card Y$ of Equation~\eqref{eq:eps}:

  \begin{enumerate}[label=(A{\arabic*})]
  \item \label{axQ:BM-abs(lt)} if $X$ is a local terminal then
    $\card{X}= 1$;
  \item \label{axQ:BM-fibres-of-identities} for all $X \in \C$ and
    $i \in \card X$, the object $(\id_X)^{-1}(i)$ is chosen terminal;
  \item \label{axQ:BM-67} for all $f \in \C / X$ and $i \in \card X$,
    we have $\card{\smash{f^{-1}(i)}} = {\card f}^{-1}(i)$, while for
    all $g\colon fg \to f$ in $\C / X$ and $i \in \card X$,
    we have $\card{\smash{g^f_{i}}} = \smash{\card{g}^{\card f}_{i}}$;
  \item \label{axQ:BM-fibres-of-tau-maps} for $Y \in \C$, we have
    $\tau_Y^{-1}(1) = Y$, and for $g \colon Z \to Y$, we have
    $g^{\tau_Y}_1 = g$;
  \item \label{axQ:BM-fibres-of-local-fibres} for
    $g \colon fg \to f$ in $\C/X$, $i \in \card X$ and
    $j \in \card{f}^{-1}(i)$, we have that
    $(g^f_i)^{-1}(j) = g^{-1}(\varepsilon j)$, and given also
    $h \colon fgh \to fg$ in $\C / X$, we have
    $(h^{fg}_i)^{g^f_i}_{j} = h^g_{\varepsilon j}$.
  \end{enumerate}

\begin{example}
  The terminal operadic category is the category $\F$ of finite ordered sets and arbitrary maps. The cardinality functor is the identity, the fibres are the `true' fibres (as in Equation~\eqref{eq:eps}).
\end{example}

\begin{example}
  For the present purposes the key example is the category $\Sur_{\text{ord}}$ of finite ordinals and arbitrary surjections. 
The cardinality functor is the inclusion functor $\Sur_{\text{ord}} \to \F$.  The fibre functor is the same as that from $\F$, but note that these fibres are not true fibres in the strict sense of the word, because they are not given by pullback.  Indeed, the category of surjections does not have pullbacks.  And the `inclusion of a fibre' is not a morphism in the category.  It is important nevertheless that many constructions with surjections can be interpreted as taking place in $\F$. 
  The axioms are easily verified.
\end{example}

\paragraph{The construction}

We now work with $\Set$-valued species as in the  classical theory.  This is needed to achieve the strictness  characteristic for operadic categories.  We also need to assume that the hereditary species have the property that $H[1] = 1$. Schmitt~\cite{Schmitt} calls such hereditary species \emph{simple}. This is true for example for simple graphs.
  
  Given a simple hereditary species $H\colon \Sp \to \Set$, we consider first its Grothendieck construction. It is a left fibration (discrete opfibration) $\int H \to \Sp$.
  The objects of $\int H$ are pairs $(n,x)$ where $n\in \Sp$ and $x\in H[n]$. We will denote such an object $X$. The morphisms in $\int H$ are described in the usual way.  They have an underlying span as in $\Sp$. We are interested in a subcategory, namely the subcategory obtained by pullback along the inclusion $\Sur_\text{ord} \to \Sp$ (from the category of finite ordinals and genuine surjections, not all partial surjections). We denote this category by $\mathcal H$.  Its objects are $X=(n,x)$ as before, and an arrow from $Y=(m,y)$ to $X=(n,x)$ is given by a genuine surjection $s\colon m \surj n$ such that $H[s](y)=x$.

We now work towards equipping $\mathcal{H}$ with the structure of operadic category.
The category $\mathcal{H}$ is clearly connected.  So to choose local terminal objects is to choose a global terminal object. By our assumption $H[1]=1$, there is a unique such, namely $(1,1)$, easily seen to be terminal in $\mathcal{H}$.

We define the cardinality functor to be the composite functor $\mathcal{H} \to \Sur_{\text{ord}} \to \F$.

We define the fibre functor, for each $X = (n,x) \in \mathcal{H}$ and each $i\in n$, to be the assignment
  \begin{eqnarray*}
    \mathcal{H}_{/X} & \longrightarrow & \mathcal{H}  \\
     f\colon Y \to X &   \longmapsto   & Y_i := (\card{f}^{-1}(i), 
	H[\varepsilon_i](y)).
  \end{eqnarray*}
Since this will be needed in all the checks, let us spell this out in   more detail.  We assume $X=(n,x)$ and $Y=(m,y)$, and the morphism $f\colon Y \to X$ is given by an underlying surjection $\card{f}\colon m \surj n$ such that $H[\card{f}](y)=x$.  The fibre $Y_i$ is defined to be the pair $(m_i,y_i)$, where $m_i = \card{f}^{-1}(i)$ is the fibre of the map in $\F$:
\begin{center}
  \begin{tikzcd}
    m_i \drpbk \ar[d, tail, "\varepsilon_i"'] \ar[r, two heads] & \{i\} \ar[d, tail]\\ 
    m \ar[r, two heads, "\card{f}"'] & n 
  \end{tikzcd}
\end{center}
  and $y_i$ is defined as $y_i = H[\varepsilon_i](y)$, the restriction of the $H$-structure $y$ along the injection $\varepsilon_i \colon m_i \rightarrowtail m$.
  
  We must also provide the assignment on arrows.  So given 
  \[
  Z \stackrel{g}\to  Y \stackrel{f}\to X
  \]
  considered as a morphism in $\mathcal{H}_{/X}$ from $fg$ to $f$, we need to
  provide a morphism
  \[
  Z_i \to Y_i.
  \]
  If we let $Z=(\ell,z)$ then we have $H[\card{g}](z) = y$. The morphism must be constituted by a surjection $g_i\colon \ell_i \surj m_i$, such that $H[g_i] (z_i) = y_i$, where $z_i = H[\varepsilon_i](z)$ is the point in $H[\ell_i]$ representing $Z_i$ (that is, restriction of the $H$-structure $z$ along the injection $\varepsilon_i \colon \ell_i \rightarrowtail \ell$), and $y_i =  H[\varepsilon_i](y)$ is the point in $H[m_i]$ representing $Y_i$.
  For the surjection $g_i \colon \ell_i \surj m_i$ to be valid, we need to check that $H[g_i] (z_i) = y_i$.  But this is precisely the pull-push formula for hereditary species on the pullback square from \eqref{pbkeps}:
  \begin{center}
    \begin{tikzcd}
	  \ell_i \drpbk \ar[r, two heads, "g_i"] \ar[d, tail, "\varepsilon_i"'] & m_i \ar[d, tail] \\
	  \ell \ar[r, two heads, "g"'] & m .
	  \end{tikzcd}
  \end{center}
(This shows that the assignment extends to arrows.  The check that this assignment on arrows respects composition and identity arrows is routine, and depends on transitivity of pullbacks in the skeletal category $\F$.)

  \bigskip
  
  We have now exhibited all the data required for an operadic category.
  
\begin{proposition}
  The structures on $\mathcal{H}$
  given above satisfy the axioms for an operadic 
  category.
\end{proposition}

Since this is a new class of operadic categories, not considered 
previously, and since the operadic category axioms can be a bit 
subtle, we include the details of the checks.

\begin{proof}
\begin{enumerate}[label=(A{\arabic*})]
  \item  The chosen terminal $(1,1)$ clearly has cardinality $1$.  

  \item We must check that all fibres of an identity map $\id_X \colon (n,x) \to (n,x)$ are the chosen terminal. By definition, for $i\in n$, the fibre is $(1,?)$ where $1$ is the $\F$-fibre of the identity $n \to n$, and $?$ can be no other than $1\in H[1]$.

    \item We need to compute the cardinality of a fibre $Y_i = (m_i, y_i)$ of a morphism $Y \to X$ and $i\in \card{X}=n$. But by construction this is $m_i$, the fibre of the underlying surjection $m \surj n$. We must also verify that for a triangle $Z \to Y \to X$, the cardinality of the fibre map (over $i\in \card{X}=n$) is the fibre map in $\F$. But this is clear from the definition of fibre map: it was defined to have as underlying surjection $\ell_i \surj m_i$, the fibre map in $\F$.
  
  \item We must check that for any object $Y=(m,y)$, the fibre of the unique map $(m,y)\to (1,1)$ has unique fibre $Y$. For the underlying map in $\F$ this is clear: the unique fibre of $m \surj 1$ is $m$. And the new point must be $H[\id](y)=y$, so altogether we find $Y$ again as required. We must also check that given $g\colon Z \to Y$ (given by $(\ell,z) \to (m,y)$), then the fibre map $g_1^\tau \colon Z_1 \to Y_1$ over the unique point in $(1,1)$ coincides with $g$ itself. On the $\F$-level, this is clear, as we get $\ell\surj m$ again. The points $z_i \in H[\ell]$ and $y_i \in H[m]$ are given, by construction of the fibre functor, by contravariant functoriality in the injections (fibre inclusions in $\F$) $\ell_i \rightarrowtail \ell$ and $m_i \rightarrowtail m$. But these are the identity maps, so $z_i=z$ and $y_i=y$ as required. 
  Note that axiom (A4) just says that the fibre functor $\fibre_{1,1} \colon \C_{/1} \to \C$ must coincide with the canonical projection functor.
    
  \item Given morphisms $Z\stackrel{g}\to Y \stackrel{f}\to X$ and elements $i\in \card{X}$ and $j \in \card{f}^{-1}(i) = m_i$, we need to establish that $(g^f_i)^{-1}(j) = g^{-1}(\varepsilon j)$.  In detail, if the objects and maps are given by
  \[
  (\ell,z) \stackrel{g}\longrightarrow (m,y) 
  \stackrel{f}\longrightarrow (n,x)
   \]
  and we have $i\in \card{X}=n$ and $j \in m_i$,  then we first form the diagram of pullbacks in $\F$:
  \begin{center}
	\begin{tikzcd}[sep=large]
	  (\ell_i)_j \drpbk \ar[r, twoheadrightarrow] \ar[d, 
	  rightarrowtail, "\varepsilon_j"'] 
	  \ar[dd, rightarrowtail, shift right = 8pt, bend right=42pt, "\varepsilon_{ij}"']
	  & \{j\} \ar[d, rightarrowtail] & \\
	  \ell_i \drpbk \ar[d, rightarrowtail, "\varepsilon_i"'] \ar[r, 
	  twoheadrightarrow, "\card{g}_i"] 
	  & m_i \drpbk \ar[d, rightarrowtail] \ar[r, twoheadrightarrow] 
	  & \{i\} \ar[d, rightarrowtail]\\ 
	  \ell \ar[r, twoheadrightarrow, "\card{g}"'] & m \ar[r, 
	  twoheadrightarrow, "\card{f}"'] & n.
	\end{tikzcd}
  \end{center}
  Note that the set $(\ell_i)_j$ has two interpretations: it is at the same time the fibre of $\card{g}$ over $\varepsilon j$, and the fibre of $\card{g}_i$ over $j$.  This shows that the two objects $(g^f_i)^{-1}(j)$ and $g^{-1}(\varepsilon j)$ have the same underlying set.  We just need to check their $H$-structures are the same. According to the definitions, the point in $(g_i^f)^{-1}(j)$  is given by $H[\varepsilon_{j}](z_i)$, where $z_i = H[\varepsilon_i](z)$.  On the other hand, the point in $g^{-1}(\varepsilon j)$ is given by $H[\varepsilon_{ij}](z)$. But these two are the same, by contravariant functoriality of $H$ in injections:
  \[
  H[\varepsilon_{j}](z_i) = 
  H[\varepsilon_{j}]\big(H[\varepsilon_i](z)\big) =
  H[\varepsilon_i \circ \varepsilon_j](z) = H[\varepsilon_{ij}](z) .
  \]
  
  For the second part of (A5), given morphisms $W \stackrel{h}\to Z\stackrel{g}\to Y \stackrel{f}\to X$ and elements $i\in \card{X}$ and $j \in \card{f}^{-1}(i) = m_i$, we need to establish that $(h^{fg}_i)^{g^f_i}_{j} = h^g_{\varepsilon j}$.
  These morphisms have the same source and target thanks to the first item in (A5). More precisely the second part of the axiom can be formulated as saying that this square commutes:
  \begin{center}
    \begin{tikzcd}
	  ((gh)_i^f)^{-1}(j) 
	  \ar[d, equal] 
	  \ar[r, "{(h_i^{fg})^{g_i^f}}"] 
	  &
	  (g_i^f)^{-1}(j) \ar[d, equal] \\
	  (gh)^{-1}(\varepsilon j) \ar[r, "h_{\varepsilon j}^g"'] & 
	  g^{-1}(\varepsilon j).
	\end{tikzcd}
  \end{center}
  Checking this is only a question of unpacking.  At the level of underlying sets, we have the pullback diagram
  \begin{center}
	\begin{tikzcd}[sep=large]
	  (k_i)_j \drpbk \ar[r, twoheadrightarrow] \ar[d, 
	  rightarrowtail, "\varepsilon_j"'] 
	  &
	  (\ell_i)_j \drpbk \ar[r, twoheadrightarrow] \ar[d, 
	  rightarrowtail, "\varepsilon_j"'] 
	  & \{j\} \ar[d, rightarrowtail] & \\
	  k_i \drpbk \ar[d, rightarrowtail, "\varepsilon_i"'] \ar[r, 
	  twoheadrightarrow, "\card{h}_i"]&
	  \ell_i \drpbk \ar[d, rightarrowtail, "\varepsilon_i"'] \ar[r, 
	  twoheadrightarrow, "\card{g}_i"] 
	  & m_i \drpbk \ar[d, rightarrowtail] \ar[r, twoheadrightarrow] 
	  & \{i\} \ar[d, rightarrowtail]\\ 
	  k \ar[r, twoheadrightarrow, "\card{h}"'] &
	  \ell \ar[r, twoheadrightarrow, "\card{g}"'] & m \ar[r, 
	  twoheadrightarrow, "\card{f}"'] & n.
	\end{tikzcd}
  \end{center}
  The point is that the surjection $(k_i)_j \surj (\ell_i)_j$ has two interpretations, namely as the $j$-fibre map of $\card{h}_i$ or as the $\varepsilon j$-fibre map of $\card{h}$.  But this is precisely to say that the two morphisms $(h^{fg}_i)^{g^f_i}_{j}$ and $h^g_{\varepsilon j}$ have the same underlying surjection.  But they also have the same source (and the same target), by the first part of A5. It follows that they agree, because the underlying map is a surjection and the projection $\mathcal{H} \to \Sur_\text{ord}$ is a discrete opfibration by construction.
\end{enumerate}
\end{proof}

\paragraph{Operadic functors}

  A functor $F\colon \C \to \D$ between operadic categories is called an \emph{operadic functor} if it strictly preserves local terminal objects, strictly commutes with the cardinality functors to $\F$, and preserves fibres and fibre maps in the sense that
  \begin{equation*}
    F(f^{-1}(i)) = (Ff)^{-1}(i) \qquad \text{and} \qquad F(g^f_{i}) =
    (Fg)^{Ff}_{i}
  \end{equation*}
  for all $g\colon fg \to f$ in $\C/X$ and $i \in \card X$.
  We denote by $\kat{OpCat}$ the category of operadic categories and operadic functors.

\begin{proposition}
  The construction given above is the object part of a functor $\kat{HSp}_{\operatorname{simple}} \to \kat{OpCat}$.
\end{proposition}

\begin{proof}
  A morphism of hereditary species is by definition a natural transformation $F\colon H' \Rightarrow H$, or equivalently a morphism  of discrete opfibrations over $\Sp$.
  Clearly this defines also a morphism $F\colon \mathcal{H}' \to \mathcal{H}$  of discrete opfibrations over $\Sur_\text{ord}$, and in particular a functor. We just need to check that this functor is operadic.  It is clear that it preserves the chosen terminal objects.  It is also clear that it is compatible with cardinality, since a morphism of discrete opfibrations over $\Sur_\text{ord}$ obviously induces a functor over $\F$.  To check compatibility with fibres, consider a morphism $f\colon (m,y) \to (n,x)$ in $\mathcal{H}'$ and an element $i\in n$.  The fibre over $i$ is by definition $(m_i, H'[\varepsilon_i](y))$, and applying $F$ to that gives
  \[
  F(f^{-1}(i)) = F \big( m_i, H'[\varepsilon_i](y)\big) 
  = \big( m_i, F(H'[\varepsilon_i](y) \big)
  = \big( m_i, H[\varepsilon_i](F(y)) \big),
  \]
  the last equality by naturality of $F$ with respect to the arrow  $\varepsilon_i \colon m_i \rightarrowtail m$.  But the last object is precisely the
  fibre of $F(f)$ over $i$, as required.
  
  We also have to show that given 
  $(\ell,z) \stackrel{g}\to (m,y) \stackrel{f}\to (n,x)$ 
  and an element $i\in n$, we have
  \[
    F( g_i^f ) = ( F(g))_i^{Ff} .
  \]
  This is well typed in view of the first part of the proof. More  precisely the assertion is that this diagram commutes:
  \begin{center}
    \begin{tikzcd}[column sep=large]
	  F(fg)^{-1}(i) 
	  \ar[d, equal] 
	  \ar[r, "F(g_i^{f})"] 
	  &
	  F(f^{-1}(i)) \ar[d, equal] \\
	  F(fg)^{-1}(i) \ar[r, "F(g)_i^{Ff}"'] & 
	  F(f)^{-1}(i).
	\end{tikzcd}
  \end{center}
  Checking this is only a question of unpacking.  Both morphisms have the same underlying surjection, namely $\card{F(g)}_i\colon \ell_i \surj m_i$.
  Since they also have the same source (and target), the fact that  $\mathcal{H} \to \Sur_\text{ord}$ is a discrete opfibration ensures that they are also equal as morphisms in $\mathcal{H}$, as required.
\end{proof}

\begin{example}
  Let $H\colon \Sp \to \Set$ be the hereditary species of simple graphs (see  Example~\ref{graphs}).  The associated operadic category is the category  $\mathcal{H}$ whose objects are simple graphs (with vertex set some  ordinal $n$), and whose morphisms are graph contractions $G \surj Q$.
  This means it is a surjective map on vertices, and also on edges, and edges are allowed to map to a vertex.  The chosen terminal graph is the one-vertex graph. The cardinality of a graph is the set of vertices, and the fibre of a contraction $G \surj Q$ over some vertex in $Q$ is the preimage of that vertex (the graph contracted onto the vertex).
\end{example}

\newpage


\begin{thebibliography}{10}

\bibitem{Abe}
E.~Abe.
\newblock {\em Hopf algebras}, volume~74 of {\em Cambridge Tracts in
  Mathematics}.
\newblock Cambridge University Press, Cambridge-New York, 1980.

\bibitem{AM}
M.~Aguiar and S.~Mahajan.
\newblock {\em Monoidal functors, species and {H}opf algebras}, volume~29 of
  {\em CRM Monograph Series}.
\newblock American Mathematical Society, Providence, RI, 2010.

\bibitem{BD}
J.~C. Baez and J.~Dolan.
\newblock From finite sets to {F}eynman diagrams.
\newblock In {\em Mathematics unlimited---2001 and beyond}, pages 29--50.
  Springer, Berlin, 2001.

\bibitem{BM}
M.~Batanin and M.~Markl.
\newblock Operadic categories and duoidal {D}eligne's conjecture.
\newblock {\em Advances in Mathematics}, 285:1630--1687, 2015.

\bibitem{BHZ}
Y.~Bruned, M.~Hairer, and L.~Zambotti.
\newblock Algebraic renormalisation of regularity structures.
\newblock {\em Inventiones Mathematicae}, 2018.

\bibitem{CEFM}
D.~Calaque, K.~Ebrahimi-Fard, and D.~Manchon.
\newblock Two interacting {H}opf algebras of trees: a {H}opf-algebraic approach
  to composition and substitution of {B}-series.
\newblock {\em Advances in Applied Mathematics}, 47:282--308, 2011.

\bibitem{Ca:bicomodules}
L.~Carlier.
\newblock Incidence bicomodules, {M}öbius inversion, and a {R}ota formula for
  infinity adjunctions.
\newblock Preprint, \arxiv{1801.07504}.

\bibitem{CK:HTCG}
L.~Carlier and J.~Kock.
\newblock Homotopy theory and combinatorics of groupoids.
\newblock In preparation.

\bibitem{CF}
P.~Cartier and D.~Foata.
\newblock {\em Problèmes combinatoires de commutation et réarrangements}.
\newblock Lecture Notes in Mathematics, No. 85. Springer-Verlag, Berlin-New
  York, 1969.

\bibitem{DK}
T.~Dyckerhoff and M.~Kapranov.
\newblock Higher {S}egal spaces {I}.
\newblock Preprint, \arxiv{1212.3563}. To appear in \emph{Springer Lecture
  Notes}.

\bibitem{Foissy:bialgebragraphs}
L.~Foissy.
\newblock Chromatic polynomials and bialgebras of graphs.
\newblock Preprint, \arxiv{1611.04303}.

\bibitem{Foissy:operads}
L.~Foissy.
\newblock Algebraic structures associated to operads.
\newblock Preprint, \arxiv{1702.05344}.

\bibitem{GKT:HLA}
I.~Gálvez-Carrillo, J.~Kock, and A.~Tonks.
\newblock Homotopy linear algebra.
\newblock {\em Proceedings of the Royal Society of Edinburgh: Section A
  Mathematics}, 148(2):293--325, 2018.

\bibitem{GKT1}
I.~Gálvez-Carrillo, J.~Kock, and A.~Tonks.
\newblock Decomposition spaces, incidence algebras and {M}öbius inversion {I}:
  basic theory.
\newblock {\em Advances in Mathematics}, 331:952--1015, 2018.

\bibitem{GKT2}
I.~Gálvez-Carrillo, J.~Kock, and A.~Tonks.
\newblock Decomposition spaces, incidence algebras and {M}öbius inversion
  {II}: completeness, length filtration, and finiteness.
\newblock {\em Advances in Mathematics}, 333:1242--1292, 2018.

\bibitem{GKT:restrict}
I.~Gálvez-Carrillo, J.~Kock, and A.~Tonks.
\newblock Decomposition spaces and restriction species.
\newblock {\em International Mathematics Research Notices}, 2018.
\newblock
  \href{https://academic.oup.com/imrn/advance-article-abstract/doi/10.1093/imrn/rny089/5095270}{doi:10.1093/imrn/rny089}.

\bibitem{GKT:combinatorics}
I.~Gálvez-Carrillo, J.~Kock, and A.~Tonks.
\newblock Decomposition spaces in combinatorics.
\newblock Preprint, \arxiv{1612.09225}.

\bibitem{GKW}
R.~Garner, J.~Kock, and M.~Weber.
\newblock Operadic categories and décalage.
\newblock Preprint, \arxiv{1812.01750}.

\bibitem{GHK}
D.~Gepner, R.~Haugseng, and J.~Kock.
\newblock $\infty$-operads as analytic monads.
\newblock Preprint, \arxiv{1712.06469}.


\bibitem{Joyal:species}
A.~Joyal.
\newblock Une théorie combinatoire des séries formelles.
\newblock {\em Advances in Mathematics}, 42:1--82, 1981.

\bibitem{Kock:BD-bis}
J.~Kock.
\newblock The incidence comodule bialgebra of the {B}aez--{D}olan construction.
\newblock Talk at \emph{Rencontres du GDR Renormalisation}, Clermont-Ferrand,
  2018.

\bibitem{KW}
J.~Kock and M.~Weber.
\newblock Faà di {B}runo for operads and internal algebras.
\newblock {\em Journal of the London Mathematical Society}, 2018.
\newblock
  \href{https://londmathsoc.onlinelibrary.wiley.com/doi/10.1112/jlms.12201}{doi:10.1112/jlms.12201}.

\bibitem{Leroux76}
P.~Leroux.
\newblock Les catégories de {M}öbius.
\newblock {\em Cahiers de topologie et géométrie différentielle},
  16:280--282, 1975.

\bibitem{Lurie}
J.~Lurie.
\newblock {\em {Higher Topos Theory}}.
\newblock Annals of Mathematics Studies. Princeton University Press, 2009.
\newblock Available at \url{http://www.math.harvard.edu/~lurie/papers/HTT.pdf}.

\bibitem{Rota}
G.-C. Rota.
\newblock On the foundations of combinatorial theory {I}. {T}heory of {M}öbius
  functions.
\newblock {\em Zeitschrift für Wahrscheinlichkeitstheorie und Verwandte
  Gebiete}, 2:340--368, 1964.

\bibitem{Schmitt}
W.~R. Schmitt.
\newblock {H}opf algebras of combinatorial structures.
\newblock {\em Canadian Journal of Mathematics}, 45:412--428, 1993.

\bibitem{Walde}
T.~Walde.
\newblock Hall monoidal categories and categorical modules.
\newblock Preprint, \arxiv{1611.08241}.

\bibitem{Young}
M.~B. {Young}.
\newblock Relative $2$-{S}egal spaces.
\newblock {\em Algebraic and Geometric Topology}, 18:975--1039, 2018.

\end{thebibliography}

\address
\end{document}